\documentclass[a4paper,oneside,10.5pt]{article}%
\usepackage{amsmath}
\usepackage{amsfonts}
\usepackage{amssymb}
\usepackage{graphicx}%
\providecommand{\U}[1]{\protect\rule{.1in}{.1in}}

\newtheorem{theorem}{Theorem}[section]

\newtheorem{corollary}[theorem]{Corollary}

\newtheorem{definition}[theorem]{Definition}
\newtheorem{assumption}[theorem]{Assumption}

\newtheorem{lemma}[theorem]{Lemma}

\newtheorem{remark}[theorem]{Remark}

\newenvironment{proof}[1][Proof]{\noindent \textbf{#1.} }{\  \rule{0.5em}{0.5em}}
\numberwithin{equation}{section}

\begin{document}

\title{Solvability of one kind of forward-backward stochastic difference equations}
\author{Shaolin Ji\thanks{Shandong university-Zhongtai Securities Institute for
Financial Studies, Shandong University, Jinan, Shandong, 250100, China, Email:
jsl@sdu.edu.cn. This work was supported by National Key R\&D Program of China
(NO. 2018YFA0703900) and National Natural Science Foundation of China (No.
11971263; 11871458).}
\and Haodong Liu\thanks{School of economics, Ocean University of China, Qingdao,
Shandong 266100, PR China. (liuhaodong@ouc.edu.cn). } }
\date{}
\maketitle

\textbf{Abstract}: In this paper, we study the solvability problem for one
kind of fully coupled forward-backward stochastic difference equations
(FBS$\Delta$Es). With the help of the necessary and sufficient condition for
the solvability of the linear FBS$\Delta$Es, under the monotone assumption, we
obtain the existence and uniqueness theorem for the general nonlinear ones.


\addcontentsline{toc}{section}{\hspace*{1.8em}Abstract}

\section{Introduction}

It is well-known that forward-backward stochastic differential equations
(FBSDEs) are widely studied by many researchers and there are fruitful results
in both theory and applications (see \cite{hp95} and \cite{mpy94, mwzz15,
my93, pt99}).

As the discrete time counterpart of FBSDEs, FBS$\Delta$Es also have wide
applications in many areas, such as discrete time stochastic optimal control
theory, mathematical finance, numerical calculation, etc. For instance, the
Hamiltonian system of a discrete time stochastic optimal control problem is a
FBS$\Delta$E (see \cite{lz15}); The FBS$\Delta$E can also be regarded as the
state equation for a discrete time recursive utility optimization problem
since some discrete time nonlinear expectations are defined by backward
stochastic difference equations (BS$\Delta$Es) (see \cite{esc15}). Moreover,
when we study the numerical solution of a FBSDE, we usually obtain a
FBS$\Delta$E through time discretization (see \cite{bz08}, \cite{cs13},
\cite{dm06}, \cite{glw05}, \cite{gp14}, \cite{mpmt02}, \cite{z04}). As far as
we know, there are few works dealing with the solvability of FBS$\Delta$Es. In
this paper, our goal is to study the solvability of the fully coupled
FBS$\Delta$Es.

The starting point of exploring the solvability of FBS$\Delta$Es is how to
formulate FBS$\Delta$Es. It is worth pointing out that even though the
BS$\Delta$E is the discrete time counterpart of the backward stochastic
differential equation (BSDE), there exist various formulations for BS$\Delta
$Es. Based on the driving process, there are mainly two formulations of
BS$\Delta$Es (see \cite{bcc14,cs13,ce10+,ce11}). One is driving by a finite
state process taking values from the basis vectors as in \cite{ac16,ce10+} and
the other is driving by a martingale with independent increments as in
\cite{bcc14,cs13}. In this paper, we take the second formulation to construct
our FBS$\Delta$Es. As for the numerical solutions of the FBSDEs, the authors
in \cite{bz08} and \cite{dm06} mentioned the solvability of the corresponding
time discretization equations with Markovian assumptions and degenerate
diffusion terms (independence of $Z$).

Before investigating the solvability of FBS$\Delta$Es, as preliminaries, we
first prove some results about the BS$\Delta$Es. In more details, we formulate
one kind of BS$\Delta$Es in our context, obtain the explicit representation of
the solution $\left(  Y,Z,N\right)  $ and prove the existence and uniqueness
of solutions to our formulated BS$\Delta$E.

Then we study the solvability problem for the linear FBS$\Delta$Es. The linear
FBS$\Delta$Es can be transformed into $2m$-dimensional linear algebraic
equations of $\mathbb{E}\left[  X_{t+1}|\mathcal{F}_{t}\right]  $ and
$\mathbb{E}\left[  X_{t+1}\Delta W_{t}|\mathcal{F}_{t}\right]  $. And we prove
the equivalence between the solvability of linear FBS$\Delta$Es and the
solvability of linear algebraic equations which leads to a necessary and
sufficient condition for the existence and uniqueness of solutions. By solving
the linear algebraic equations, we can decouple the forward and backward
variables and obtain the explicit expressions of $Y_{t}$, $Z_{t}$ with respect
to $X_{t}$. Thus, the linear FBS$\Delta$Es can be solved recursively in an
explicit form.

With the help of the obtained results for linear FBS$\Delta$Es, the
solvability problem for the nonlinear ones is studied. We apply the
continuation method developed in \cite{hp95, pw99, y97} to solve our
FBS$\Delta$Es. Under the monotone assumption, we obtain the existence and
uniqueness theorem for the general nonlinear FBS$\Delta$Es.

The analysis of stochastic difference equations is quite different from that
of the stochastic differential equations. For example, in the continuous-time
case, when using It\^{o} formula, the product of the drift term will
disappear, which can be formally understood as $dt\cdot dt=0$. However, since
\[
\Delta\left\langle X_{t},Y_{t}\right\rangle =\left\langle \Delta X_{t}%
,Y_{t}\right\rangle +\left\langle X_{t},\Delta Y_{t}\right\rangle
+\left\langle \Delta X_{t},\Delta Y_{t}\right\rangle
\]
in the discrete-time case, the product of the drift term will not disappear,
which makes it difficult to applying the continuation method directly in our
discrete time context. In order to overcome this difficulty, we adopt the form
of the product rule%
\begin{equation}
\Delta\left\langle X_{t},Y_{t}\right\rangle =\left\langle X_{t+1},\Delta
Y_{t}\right\rangle +\left\langle \Delta X_{t},Y_{t}\right\rangle
\label{ito formula}%
\end{equation}
which eliminates the drift term. To apply the monotone condition, the terms
$X_{t+1}$ and $\Delta Y_{t}$ in (\ref{ito formula}) should have the same time
subscript. So we formulate that the generator $f$ of the BS$\Delta$E in
(\ref{nonlinear_fbsde_c_method_0}) should only depend on the solution of
(\ref{nonlinear_fbsde_c_method_0}) at time $t+1$. Thus, the continuation
method can be implemented in the discrete time context. It is worth pointing
out that our formulation of BS$\Delta$Es is just the formulation of the
adjoint equations for discrete-time stochastic optimal control problems (see
\cite{lz15}).

The remainder of this paper is organized as follows. In Section 2 we present
preliminary results of the BS$\Delta$Es and formulate FBS$\Delta$Es. Then we
obtain the explicit solutions for linear FBS$\Delta$Es
(\ref{fbsde_linear_1d_4}) in Section 3. The solvability results for nonlinear
FBS$\Delta$Es (\ref{nonlinear_fbsde_c_method_0}) are given in Section 4. In
the appendix, we give the proofs of two theorems related to BS$\Delta$Es.

\section{Preliminaries and problem formulation}

Let $T$ be a deterministic terminal time and $\mathcal{T}:=\left\{
0,1,...,T\right\}  $. Consider a filtered probability space $\left(
\Omega,\mathcal{F},\left\{  \mathcal{F}_{t}\right\}  _{0\leq t\leq
T},P\right)  $ with $t\in\mathcal{T}$, $\mathcal{F}_{0}=\left\{
\emptyset,\Omega\right\}  $ and $\mathcal{F=F}_{T}$. For a $\mathcal{F}_{t}%
$-adapted process $(U_{t})$, define the difference operator $\Delta$ as
$\Delta U_{t}=U_{t+1}-U_{t}$. Let $W$ be a fixed $\mathbb{R}^{d}$-valued
square integrable martingale process with independent increments, i.e.
$\mathbb{E}\left[  \Delta W_{t}|\mathcal{F}_{t}\right]  =\mathbb{E}\left[
\Delta W_{t}\right]  =0$ and $\mathbb{E}\left[  \Delta W_{t}\left(  \Delta
W_{t}\right)  ^{\ast}\right]  =I_{d}$ for any $t\in\left\{  0,...,T-1\right\}
$ where $\left(  \cdot\right)  ^{\ast}$ denotes vector transposition. We
assume that $\mathcal{F}_{t}$ is the completion of the $\sigma$-algebra
generated by the process $W$ up to time $t$.

Denote by $L^{2}\left(  \mathcal{F}_{t};\mathbb{R}^{n}\right)  $ the set of
all $\mathcal{F}_{t}-$measurable square integrable random variable $X_{t}$
taking values in $\mathbb{R}^{n}$ and by $\mathcal{M}^{2}\left(
0,t;\mathbb{R}^{n}\right)  $ the set of all $\left\{  \mathcal{F}_{s}\right\}
_{0\leq s\leq t}$-adapted square integrable process $X$ taking values in
$\mathbb{R}^{n}$. Moreover, we define $e_{i}=\left(
0,0,...,0,1,0,...,0\right)  ^{\ast}\in\mathbb{R}^{n}$ and mention that an
inequality on a vector quantity is to hold componentwise.

\subsection{Solvability and estimates results for BS$\Delta$E}

Consider the following backward stochastic difference equation (BS$\Delta$E):%

\begin{equation}
\left\{
\begin{array}
[c]{rcl}%
\Delta Y_{t} & = & -f\left(  t+1,Y_{t+1},Z_{t+1}\right)  +Z_{t}\Delta
W_{t}+\Delta N_{t},\\
Y_{T} & = & \eta,
\end{array}
\right.  \label{bsde_form1}%
\end{equation}
where $\eta\in L^{2}\left(  \mathcal{F}_{T};\mathbb{R}^{n}\right)  $,
$f:\Omega\times\left\{  1,2,...,T\right\}  \times\mathbb{R}^{n}\times
\mathbb{R}^{n\times d}\longmapsto\mathbb{R}^{n}$.

\begin{assumption}
\label{generator_assumption}A1. The function $f\left(  t,y,z\right)  $ is
uniformly Lipschitz continuous and independent of $z$ at $t=T$, i.e. there
exists constants $c_{1},c_{2}>0$, such that for any $t\in\left\{
1,2,...,T-1\right\}  $, $y_{1},y_{2}\in\mathbb{R}^{n}$, $z_{1},z_{2}%
\in\mathbb{R}^{n\times d}$,%
\begin{align*}
\left\vert f\left(  T,y_{1},z_{1}\right)  -f\left(  T,y_{2},z_{2}\right)
\right\vert  &  \leq c_{1}\left\vert y_{1}-y_{2}\right\vert ,\\
\left\vert f\left(  t,y_{1},z_{1}\right)  -f\left(  t,y_{2},z_{2}\right)
\right\vert  &  \leq c_{1}\left\vert y_{1}-y_{2}\right\vert +c_{2}\left\Vert
z_{1}-z_{2}\right\Vert ,\text{ }P-a.s.
\end{align*}

A2. $f\left(  t,0,0\right)  \in L^{2}\left(  \mathcal{F}_{t};\mathbb{R}%
^{n}\right)  $ for any $t\in\left\{  1,2,...,T\right\}  $.
\end{assumption}

\begin{remark}
The BS$\Delta$E (\ref{bsde_form1}) is analogous to the continuous time BSDE
driven by a general martingale (cf. \cite{eh97}), and the solution is a triple
of processes.
\end{remark}

\begin{definition}
A solution to BS$\Delta$E (\ref{bsde_form1}) is a triple of processes $\left(
Y,Z,N\right)  \in\mathcal{M}^{2}\left(  0,T;\mathbb{R}^{n}\right)
\times\mathcal{M}^{2}\left(  0,T-1;\mathbb{R}^{n\times d}\right)
\times\mathcal{M}^{2}\left(  0,T;\mathbb{R}^{n}\right)  $ which satisfies
equality (\ref{bsde_form1}) for all $t\in\left\{  0,1,...,T-1\right\}  $, and
$N$ is a martingale process strongly orthogonal to $W$.
\end{definition}

By using the Galtchouk-Kunita-Watanabe decomposition in \cite{bcc14}, we
obtain the following existence and uniqueness result of BS$\Delta$E
(\ref{bsde_form1}).

\begin{theorem}
\label{bsde_result_l}Suppose that Assumption (\ref{generator_assumption})
holds. Then for any terminal condition $\eta\in L^{2}\left(  \mathcal{F}%
_{T};\mathbb{R}^{n}\right)  $,\ BS$\Delta$E (\ref{bsde_form1}) has a unique
adapted solution $\left(  Y,Z,N\right)  $.
\end{theorem}

Besides, the following theorem gives the estimate for $\left(  Y,Z,N\right)  $.

\begin{theorem}
\label{bsde_result_2}Let\ Assumption (\ref{generator_assumption}) hold and
$\left(  Y,Z,N\right)  $ be a solution to BS$\Delta$E (\ref{bsde_form1}). Then
there exists a constant $C$, depending only on $c_{1}$, $c_{2}$ and $T$, such
that%
\[
\sum_{t=0}^{T-1}\mathbb{E}\left[  \left\vert Y_{t}\right\vert ^{2}+\left\Vert
Z_{t}\right\Vert ^{2}+\left\vert \Delta N_{t}\right\vert ^{2}\right]  \leq
C\mathbb{E}\left[  \left\vert \eta\right\vert ^{2}+\sum_{s=1}^{T}\left\vert
f\left(  s,0,0\right)  \right\vert ^{2}\right]  .
\]

\end{theorem}

For the convenience of readers, we put the proofs of the above two theorems in
the appendix. The following results are direct consequence of Theorem
\ref{bsde_result_2}.

\begin{corollary}
\label{bsde_result_3}For $i=1,2$, assume $\left(  \eta_{i},f_{i}\right)  $
satisfy Assumption (\ref{generator_assumption}) and $\left(  Y_{i},Z_{i}%
,N_{i}\right)  $ is a solution to BS$\Delta$E (\ref{bsde_form1})\ with
coefficients $\left(  \eta_{i},f_{i}\right)  $.\ Then there exists a constant
$C$, depending only on $c_{1}$, $c_{2}$ and $T$, such that%
\[
\sum_{t=0}^{T-1}\mathbb{E}\left[  \left\vert \widehat{Y}_{t}\right\vert
^{2}+\left\Vert \widehat{Z}_{t}\right\Vert ^{2}+\left\vert \Delta
\widehat{N}_{t}\right\vert ^{2}\right]  \leq C\mathbb{E}\left[  \left\vert
\widehat{\eta}\right\vert ^{2}+\sum_{s=1}^{T}\left\vert \widehat{f}\left(
s,Y_{1,s},Z_{1,s}\right)  \right\vert ^{2}\right]  .
\]
where%
\begin{align*}
\widehat{Y}  &  =Y_{1}-Y_{2},\widehat{Z}=Z_{1}-Z_{2},\widehat{N}=N_{1}%
-N_{2},\\
\widehat{\eta}  &  =\eta_{1}-\eta_{2},\widehat{f}=f_{1}-f_{2}.
\end{align*}

\end{corollary}

\subsection{Formulation of FBS$\Delta$E}

Now we formulate the linear FBS$\Delta$E\ and nonlinear FBS$\Delta$E. For
simplicity, in the following, we suppose $W$ is one-dimensional. The
multi-dimensional version is the same.

Consider the following linear FBS$\Delta$E:
\begin{equation}
\left\{
\begin{array}
[c]{rcl}%
\Delta X_{t} & = & A_{t}X_{t}+B_{t}Y_{t}+C_{t}Z_{t}+D_{t}+\left(  \overline
{A}_{t}X_{t}+\overline{B}_{t}Y_{t}+\overline{C}_{t}Z_{t}+\overline{D}%
_{t}\right)  \Delta W_{t},\\
\Delta Y_{t} & = & \widehat{A}_{t+1}X_{t+1}+\widehat{B}_{t+1}Y_{t+1}%
+\widehat{C}_{t+1}Z_{t+1}+\widehat{D}_{t+1}+Z_{t}\Delta W_{t}+\Delta N_{t},\\
X_{0} & = & x_{0},\\
Y_{T} & = & GX_{T}+g,
\end{array}
\right.  \label{fbsde_linear_1d_4}%
\end{equation}
where the coefficients satisfy

(i) $A,\overline{A}$\ are deterministic functions valued in $\mathbb{R}%
^{m\times m}$,\ $\widehat{A}$ are deterministic functions valued in
$\mathbb{R}^{n\times m}$, $G$ is deterministic matrix in $\mathbb{R}^{n\times
m}$;

(ii) $B,\overline{B},C,\overline{C}$\ are deterministic functions valued in
$\mathbb{R}^{m\times n}$, $\widehat{B},\widehat{C}$\ are deterministic
functions valued in $\mathbb{R}^{n\times n}$ with$\ \widehat{C}_{T}=0$;

(iii) $D,\overline{D}\in\mathcal{M}^{2}\left(  0,T-1;\mathbb{R}^{m}\right)  $,
$\widehat{D}\in\mathcal{M}^{2}\left(  1,T;\mathbb{R}^{n}\right)  $ and $g\in
L^{2}\left(  \mathcal{F}_{T};\mathbb{R}^{n}\right)  $.

\begin{remark}
Notice that the coefficients of the homogeneous terms are supposed to be
deterministic functions and the\ coefficients of the nonhomogeneous terms can
be stochastic processes. Besides,\ it can be seen that the generator of the
backward equation in (\ref{fbsde_linear_1d_4}) is independent of $Z_{T}$ at
time $T$ since $\widehat{C}_{T}=0$.
\end{remark}

Let%
\begin{align*}
b\left(  \omega,t,x,y,z\right)   &  :\Omega\times\left\{  0,1,...,T-1\right\}
\times\mathbb{R}^{m}\times\mathbb{R}^{n}\times\mathbb{\mathbb{R}}%
^{n}\mathbb{\rightarrow R}^{m},\\
\sigma\left(  \omega,t,x,y,z\right)   &  :\Omega\times\left\{
0,1,...,T-1\right\}  \times\mathbb{R}^{m}\times\mathbb{R}^{n}\times
\mathbb{\mathbb{R}}^{n}\mathbb{\rightarrow R}^{m},\\
f\left(  \omega,t,x,y,z\right)   &  :\Omega\times\left\{  1,2,...,T\right\}
\times\mathbb{R}^{m}\times\mathbb{R}^{n}\times\mathbb{R}^{n}%
\mathbb{\rightarrow R}^{n},\\
h\left(  \omega,x\right)   &  :\Omega\times\mathbb{R}^{n}\mathbb{\rightarrow
R}^{n}%
\end{align*}
and $b\left(  T,x,y,z\right)  \equiv0$, $\sigma\left(  T,x,y,z\right)
\equiv0$, $l\left(  T,x,y,z\right)  \equiv0$, $f\left(  0,x,y,z\right)
\equiv0$.

Consider the following nonlinear FBS$\Delta$E:%
\begin{equation}
\left\{
\begin{array}
[c]{rcl}%
\Delta X_{t} & = & b\left(  \omega,t,X_{t},Y_{t},Z_{t}\right)  +\sigma\left(
\omega,t,X_{t},Y_{t},Z_{t}\right)  \Delta W_{t},\\
\Delta Y_{t} & = & -f\left(  \omega,t+1,X_{t+1},Y_{t+1},Z_{t+1}\right)
+Z_{t}\Delta W_{t}+\Delta N_{t},\\
X_{0} & = & x_{0},\\
Y_{T} & = & h\left(  X_{T}\right)  .
\end{array}
\right.  \label{nonlinear_fbsde_c_method_0}%
\end{equation}

Assume that $G$ is an $n\times m$\ full-rank matrix and let%
\[
\lambda=\left(
\begin{array}
[c]{c}%
x\\
y\\
z
\end{array}
\right)  ,A\left(  t,\lambda\right)  =\left(
\begin{array}
[c]{c}%
-G^{\ast}f\\
Gb\\
G\sigma
\end{array}
\right)  \left(  t,\lambda\right)  .
\]

\begin{assumption}
\label{assu-nonlinear} (i) The coefficients are uniformly Lipschitz continuous
with respect to $\lambda$, and $f$ is independent of variable $z$ at time
$T$,\ i.e. there exists $c>0$ such that for any $t\in\left\{
0,1,...,T-1\right\}  $, $P-a.s.,$%
\begin{align*}
\left\vert A\left(  t,\lambda\right)  -A\left(  t,\lambda^{^{\prime}}\right)
\right\vert  &  \leq c\left\vert \lambda-\lambda^{^{\prime}}\right\vert ,\\
\left\vert h\left(  x\right)  -h\left(  x^{^{\prime}}\right)  \right\vert  &
\leq c\left\vert x-x^{^{\prime}}\right\vert .
\end{align*}
Moreover,\ $\left\vert f\left(  \omega,T,x,y,z\right)  -f\left(
\omega,T,x^{^{\prime}},y^{^{\prime}},z^{^{\prime}}\right)  \right\vert \leq
c\left(  \left\vert x-x^{^{\prime}}\right\vert +\left\vert y-y^{^{\prime}%
}\right\vert \right)  ;$

(ii) For any $\lambda\in\mathbb{R}^{m}\times\mathbb{R}^{n}\times\mathbb{R}%
^{n}$, $A\left(  t,\lambda\right)  \in\mathcal{M}\left(  0,T-1;\mathbb{R}%
^{m}\times\mathbb{R}^{n}\times\mathbb{R}^{n}\right)  $,\ $h(x)\in L^{2}\left(
\mathcal{F}_{T}\right)  ;$

(iii) The coefficients satisfy the following monotone conditions, i.e.\ when
$t\in\left\{  1,...,T-1\right\}  $,%
\[%
\begin{array}
[c]{cl}
& \left\langle A\left(  t,\lambda\right)  -A\left(  t,\lambda^{^{\prime}%
}\right)  ,\lambda-\lambda^{^{\prime}}\right\rangle \\
\leq & -\beta_{1}\left\Vert G\left(  x-x^{^{\prime}}\right)  \right\Vert
^{2}-\beta_{2}\left\Vert G^{\ast}\left(  y-y^{^{\prime}}\right)  \right\Vert
^{2}-\beta_{2}\left\Vert G^{\ast}\left(  z-z^{^{\prime}}\right)  \right\Vert
^{2},P-a.s.;
\end{array}
\]

\end{assumption}

when\ $t=T$,%
\[
\left\langle -G^{\ast}f\left(  T,x,y,z\right)  +G^{\ast}f\left(
T,x^{^{\prime}},y^{^{\prime}},z^{^{\prime}}\right)  ,x-x^{^{\prime}%
}\right\rangle \leq-\beta_{1}\left\Vert G\left(  x-x^{^{\prime}}\right)
\right\Vert ^{2},P-a.s.;
\]

when\ $t=0$,%
\[%
\begin{array}
[c]{cl}
& \left\langle Gb\left(  0,\lambda\right)  -Gb\left(  0,\lambda^{^{\prime}%
}\right)  ,y-y^{^{\prime}}\right\rangle +\left\langle G\sigma\left(
0,\lambda\right)  -G\sigma\left(  0,\lambda^{^{\prime}}\right)  ,z-z^{^{\prime
}}\right\rangle \\
\leq & -\beta_{2}\left\Vert G^{\ast}\left(  y-y^{^{\prime}}\right)
\right\Vert ^{2}-\beta_{2}\left\Vert G^{\ast}\left(  z-z^{^{\prime}}\right)
\right\Vert ^{2},P-a.s..
\end{array}
\]

and%
\[
\left\langle h\left(  x\right)  -h\left(  x^{^{\prime}}\right)  ,G\left(
x-x^{^{\prime}}\right)  \right\rangle \geq0,P-a.s.,
\]
where $\beta_{1}$ and $\beta_{2}$ are given nonnegative constants with
$\beta_{1}+\beta_{2}>0$. Moreover we have $\beta_{1}>0$ (resp. $\beta_{2}>0$)
when $n>m$ (resp. $m>n$).

We give the definition of the solutions to FBS$\Delta$E\thinspace
(\ref{fbsde_linear_1d_4}) and FBS$\Delta$E\thinspace
(\ref{nonlinear_fbsde_c_method_0}).

\begin{definition}
$\left(  X,Y,Z,N\right)  \in\mathcal{M}^{2}\left(  0,T;\mathbb{R}^{m}\right)
\times\mathcal{M}^{2}\left(  0,T;\mathbb{R}^{n}\right)  \times\mathcal{M}%
^{2}\left(  0,T-1;\mathbb{R}^{n}\right)  \times\mathcal{M}^{2}\left(
0,T;\mathbb{R}^{n}\right)  $ is called a solution of FBS$\Delta$%
E\thinspace(\ref{fbsde_linear_1d_4})\thinspace(resp. FBS$\Delta$%
E\thinspace(\ref{nonlinear_fbsde_c_method_0})) if $\left(  X,Y,Z,N\right)  $
satisfies the difference equation (\ref{fbsde_linear_1d_4})\thinspace(resp.
(\ref{nonlinear_fbsde_c_method_0})) at any time $t$, $P-a.s.$ with $N$ to be a
martingale process strongly orthogonal to $W$.
\end{definition}

\section{Solvability of linear FBS$\Delta$Es}

In this section, we study the solvability of\ linear FBS$\Delta$%
E\ (\ref{fbsde_linear_1d_4}).

Define%
\[
\Gamma_{t}\left(  P_{t+1}\right)  =I-\left(
\begin{array}
[c]{cc}%
B_{t}P_{t+1} & C_{t}P_{t+1}\\
\overline{B}_{t}P_{t+1} & \overline{C}_{t}P_{t+1}%
\end{array}
\right)  ,
\]
where%
\[
\left\{
\begin{array}
[c]{ccl}%
P_{t} & = & -\widehat{A}_{t}+\left(  \left(  I-\widehat{B}_{t}\right)
P_{t+1},-\widehat{C}_{t}P_{t+1}\right)  \left[  \Gamma_{t}\left(
P_{t+1}\right)  \right]  ^{-1}\left(
\begin{array}
[c]{c}%
I+A_{t}\\
\overline{A}_{t}%
\end{array}
\right)  ,\\
P_{T} & = & -\widehat{A}_{T}+\left(  I-\widehat{B}_{T}\right)  G.
\end{array}
\right.
\]

\begin{theorem}
\label{solution theorem for linear fbsde_4}The linear FBS$\Delta
$E\ (\ref{fbsde_linear_1d_4})\ has a unique solution if and only if for
any\ $t\in\{0,1,2,...,T-1\}$, the $2m$-dimensional matrix $\Gamma_{t}\left(
P_{t+1}\right)  $ is invertible $P-a.s.$. Moreover, the solution to the
FBS$\Delta$E is%
\[
\left\{
\begin{array}
[c]{rcl}%
X_{t+1} & = & \left(  I,I\Delta W_{t}\right)  \left[  \Gamma_{t}\left(
P_{t+1}\right)  \right]  ^{-1}\left[  \left(
\begin{array}
[c]{c}%
I+A_{t}\\
\overline{A}_{t}%
\end{array}
\right)  X_{t}+\left(
\begin{array}
[c]{c}%
B_{t}\\
\overline{B}_{t}%
\end{array}
\right)  \mathbb{E}\left[  p_{t+1}|\mathcal{F}_{t}\right]  \right. \\
& \multicolumn{1}{r}{} & \multicolumn{1}{r}{\left.  +\left(
\begin{array}
[c]{c}%
C_{t}\\
\overline{C}_{t}%
\end{array}
\right)  \mathbb{E}\left[  p_{t+1}\left(  \Delta W_{t}\right)  |\mathcal{F}%
_{t}\right]  +\left(
\begin{array}
[c]{c}%
D_{t}\\
\overline{D}_{t}%
\end{array}
\right)  \right]  .}\\
Y_{t} & = & \mathbb{E}\left[  P_{t+1}X_{t+1}+p_{t+1}|\mathcal{F}_{t}\right]
,\\
Z_{t} & = & P_{t+1}\mathbb{E}\left[  X_{t+1}\left(  \Delta W_{t}\right)
|\mathcal{F}_{t}\right]  +\mathbb{E}\left[  p_{t+1}\Delta W_{t}|\mathcal{F}%
_{t}\right]  ,\\
X_{0} & = & x_{0},
\end{array}
\right.
\]
where%
\[
\left\{
\begin{array}
[c]{ccl}%
p_{t} & = & \left[  I-\widehat{B}_{t}+\left(  \left(  I-\widehat{B}%
_{t}\right)  P_{t+1},-\widehat{C}_{t}P_{t+1}\right)  \left[  \Gamma_{t}\left(
P_{t+1}\right)  \right]  ^{-1}\left(
\begin{array}
[c]{c}%
B_{t}\\
\overline{B}_{t}%
\end{array}
\right)  \right]  \mathbb{E}\left[  p_{t+1}|\mathcal{F}_{t}\right] \\
&  & +\left[  -\widehat{C}_{t}+\left(  \left(  I-\widehat{B}_{t}\right)
P_{t+1},-\widehat{C}_{t}P_{t+1}\right)  \right]  \left[  \Gamma_{t}\left(
P_{t+1}\right)  \right]  ^{-1}\left(
\begin{array}
[c]{c}%
C_{t}\\
\overline{C}_{t}%
\end{array}
\right)  \mathbb{E}\left[  p_{t+1}\left(  \Delta W_{t}\right)  |\mathcal{F}%
_{t}\right] \\
&  & +\left(  \left(  I-\widehat{B}_{t}\right)  P_{t+1},-\widehat{C}%
_{t}P_{t+1}\right)  \left[  \Gamma_{t}\left(  P_{t+1}\right)  \right]
^{-1}\left(
\begin{array}
[c]{c}%
D_{t}\\
\overline{D}_{t}%
\end{array}
\right)  -\widehat{D}_{t}.\\
p_{T} & = & \left(  I-\widehat{B}_{T}\right)  g-\widehat{D}_{T}.
\end{array}
\right.
\]

\end{theorem}

\begin{proof}
We first consider the difference equations at time $T$. Let
\[
\lambda_{T}=-\widehat{A}_{T}X_{T}+\left(  I-\widehat{B}_{T}\right)
Y_{T}-\widehat{D}_{T}.
\]
Since $Y_{T}=GX_{T}+g$, we have%
\begin{align*}
\lambda_{T}  &  =\left[  -\widehat{A}_{T}+\left(  I-\widehat{B}_{T}\right)
G\right]  X_{T}+\left(  I-\widehat{B}_{T}\right)  g-\widehat{D}_{T}\\
&  =P_{T}X_{T}+p_{T}.
\end{align*}
According to Theorem \ref{bsde_result_l},%
\begin{equation}
\left\{
\begin{array}
[c]{ccl}%
Y_{T-1} & = & \mathbb{E}\left[  \lambda_{T}|\mathcal{F}_{T-1}\right] \\
& = & P_{T}\mathbb{E}\left[  X_{T}|\mathcal{F}_{T-1}\right]  +\mathbb{E}%
\left[  p_{T}|\mathcal{F}_{T-1}\right]  ,\\
Z_{T-1} & = & \mathbb{E}\left[  \lambda_{T}\left(  \Delta W_{T-1}\right)
|\mathcal{F}_{T-1}\right] \\
& = & P_{T}\mathbb{E}\left[  X_{T}\left(  \Delta W_{T-1}\right)
|\mathcal{F}_{T-1}\right]  +\mathbb{E}\left[  p_{T}\left(  \Delta
W_{T-1}\right)  |\mathcal{F}_{T-1}\right]  .
\end{array}
\right.  \label{z_representation_lfbsde4}%
\end{equation}
Taking $\mathcal{F}_{T-1}$-conditional expectation on both sides of the
forward equation,%
\[%
\begin{array}
[c]{cl}
& \mathbb{E}\left[  X_{T}|\mathcal{F}_{T-1}\right] \\
= & \left(  I+A_{T-1}\right)  X_{T-1}+B_{T-1}P_{T}\mathbb{E}\left[
X_{T}|\mathcal{F}_{T-1}\right]  +B_{T-1}\mathbb{E}\left[  p_{T}|\mathcal{F}%
_{T-1}\right] \\
& +C_{T-1}P_{T}\mathbb{E}\left[  X_{T}\Delta W_{T-1}|\mathcal{F}_{T-1}\right]
+C_{T-1}\mathbb{E}\left[  p_{T}\Delta W_{T-1}|\mathcal{F}_{T-1}\right]
+D_{T-1}.
\end{array}
\]
Multiplying the forward equation by $\Delta W_{T-1}$ and taking $\mathcal{F}%
_{T-1}$-conditional expectation,%
\[%
\begin{array}
[c]{cl}
& \mathbb{E}\left[  X_{T}\Delta W_{T-1}|\mathcal{F}_{T-1}\right] \\
= & \overline{A}_{T-1}X_{T-1}+\overline{B}_{T-1}P_{T}\mathbb{E}\left[
X_{T}|\mathcal{F}_{T-1}\right]  +\overline{B}_{T-1}\mathbb{E}\left[
p_{T}|\mathcal{F}_{T-1}\right] \\
& +\overline{C}_{T-1}P_{T}\mathbb{E}\left[  X_{T}\Delta W_{T-1}|\mathcal{F}%
_{T-1}\right]  +\overline{C}_{T-1}\mathbb{E}\left[  p_{T}\Delta W_{T-1}%
|\mathcal{F}_{T-1}\right]  +\overline{D}_{T-1}.
\end{array}
\]
Then we obtain a $2m$-dimensional linear algebra equation:%
\begin{equation}%
\begin{array}
[c]{cl}
& \left(
\begin{array}
[c]{cc}%
I-B_{T-1}P_{T} & -C_{T-1}P_{T}\\
-\overline{B}_{T-1}P_{T} & I-\overline{C}_{T-1}P_{T}%
\end{array}
\right)  \left(
\begin{array}
[c]{c}%
\mathbb{E}\left[  X_{T}|\mathcal{F}_{T-1}\right] \\
\mathbb{E}\left[  X_{T}\Delta W_{T-1}|\mathcal{F}_{T-1}\right]
\end{array}
\right) \\
= & \left(
\begin{array}
[c]{c}%
I+A_{T-1}\\
\overline{A}_{T-1}%
\end{array}
\right)  X_{T-1}+\left(
\begin{array}
[c]{c}%
B_{T-1}\\
\overline{B}_{T-1}%
\end{array}
\right)  \mathbb{E}\left[  p_{T}|\mathcal{F}_{T-1}\right] \\
& +\left(
\begin{array}
[c]{c}%
C_{T-1}\\
\overline{C}_{T-1}%
\end{array}
\right)  \mathbb{E}\left[  p_{T}\Delta W_{T-1}|\mathcal{F}_{T-1}\right]
+\left(
\begin{array}
[c]{c}%
D_{T-1}\\
\overline{D}_{T-1}%
\end{array}
\right)  .
\end{array}
\label{algebraic_equation_4}%
\end{equation}

It is easy to check that if $\left(  X_{T-1},Y_{T-1},Z_{T-1}\right)  $ is a
solution of the FBS$\Delta$E (\ref{fbsde_linear_1d_4}) at time $T-1$, then
$(\mathbb{E}\left[  X_{T}|\mathcal{F}_{T-1}\right]  $,\ $\mathbb{E}\left[
X_{T}\left(  \Delta W_{T-1}\right)  |\mathcal{F}_{T-1}\right]  )$ is a
solution of the algebra equation (\ref{algebraic_equation_4}). On the other
hand, if $(\mathbb{E}\left[  X_{T}|\mathcal{F}_{T-1}\right]  $, $\mathbb{E}%
\left[  X_{T}\left(  \Delta W_{T-1}\right)  |\mathcal{F}_{T-1}\right]  )$ is a
solution of the algebra equation (\ref{algebraic_equation_4}), then $\left(
X_{T-1},Y_{T-1},Z_{T-1}\right)  $ is a solution of the FBS$\Delta$E
(\ref{fbsde_linear_1d_4}) at time $T-1$, where%
\[
\left\{
\begin{array}
[c]{ccl}%
Z_{T-1} & = & P_{T}\mathbb{E}\left[  X_{T}\Delta W_{T-1}|\mathcal{F}%
_{T-1}\right]  +\mathbb{E}\left[  p_{T}\Delta W_{T-1}|\mathcal{F}%
_{T-1}\right]  ,\\
Y_{T-1} & = & P_{T}\mathbb{E}\left[  X_{T}|\mathcal{F}_{T-1}\right]
+\mathbb{E}\left[  p_{T}|\mathcal{F}_{T-1}\right]  .
\end{array}
\right.
\]
So the FBS$\Delta$E\thinspace(\ref{fbsde_linear_1d_4}) has a unique solution
at time $T-1$ if and only if the linear equation (\ref{algebraic_equation_4})
has a unique solution, which is equivalent to say that%
\[
\Gamma_{T-1}\left(  P_{T}\right)  =I-\left(
\begin{array}
[c]{cc}%
B_{T-1}P_{T} & C_{T-1}P_{T}\\
\overline{B}_{T-1}P_{T} & \overline{C}_{T-1}P_{T}%
\end{array}
\right)
\]
is a invertible matrix. If $\Gamma_{T-1}\left(  P_{T}\right)  $ is invertible,
then the solution of (\ref{algebraic_equation_4}) is%
\[%
\begin{array}
[c]{ll}
& \left(
\begin{array}
[c]{c}%
\mathbb{E}\left[  X_{T}|\mathcal{F}_{T-1}\right] \\
\mathbb{E}\left[  X_{T}\left(  \Delta W_{T-1}\right)  |\mathcal{F}%
_{T-1}\right]
\end{array}
\right) \\
\multicolumn{1}{c}{=} & \multicolumn{1}{c}{\left[  \Gamma_{T-1}\left(
P_{T}\right)  \right]  ^{-1}\left[  \left(
\begin{array}
[c]{c}%
I+A_{T-1}\\
\overline{A}_{T-1}%
\end{array}
\right)  X_{T-1}+\left(
\begin{array}
[c]{c}%
B_{T-1}\\
\overline{B}_{T-1}%
\end{array}
\right)  \mathbb{E}\left[  p_{T}|\mathcal{F}_{T-1}\right]  \right. }\\
\multicolumn{1}{r}{} & \multicolumn{1}{r}{\left.  +\left(
\begin{array}
[c]{c}%
C_{T-1}\\
\overline{C}_{T-1}%
\end{array}
\right)  \mathbb{E}\left[  p_{T}\left(  \Delta W_{T-1}\right)  |\mathcal{F}%
_{T-1}\right]  +\left(
\begin{array}
[c]{c}%
D_{T-1}\\
\overline{D}_{T-1}%
\end{array}
\right)  \right]  .}%
\end{array}
\]
Combining with (\ref{z_representation_lfbsde4}), we deduce%
\begin{equation}
\left\{
\begin{array}
[c]{ccl}%
Y_{T-1} & = & G_{T-1}X_{T-1}+g_{T-1},\\
Z_{T-1} & = & H_{T-1}X_{T-1}+h_{T-1},
\end{array}
\right.  \label{c3_eq_w_fbsde_linear_2_yz}%
\end{equation}
where,%
\[
\left\{
\begin{array}
[c]{ccl}%
G_{T-1} & = & \left(  P_{T},0\right)  \left[  \Gamma_{T-1}\left(
P_{T}\right)  \right]  ^{-1}\left(
\begin{array}
[c]{c}%
I+A_{T-1}\\
\overline{A}_{T-1}%
\end{array}
\right)  ,\\
H_{T-1} & = & \left(  0,P_{T}\right)  \left[  \Gamma_{T-1}\left(
P_{T}\right)  \right]  ^{-1}\left(
\begin{array}
[c]{c}%
I+A_{T-1}\\
\overline{A}_{T-1}%
\end{array}
\right)  ,\\
g_{T-1} & = & \left[  I+\left(  P_{T},0\right)  \left[  \Gamma_{T-1}\left(
P_{T}\right)  \right]  ^{-1}\left(
\begin{array}
[c]{c}%
B_{T-1}\\
\overline{B}_{T-1}%
\end{array}
\right)  \right]  \mathbb{E}\left[  p_{T}|\mathcal{F}_{T-1}\right] \\
&  & +\left(  P_{T},0\right)  \left[  \Gamma_{T-1}\left(  P_{T}\right)
\right]  ^{-1}\left[  \left(
\begin{array}
[c]{c}%
C_{T-1}\\
\overline{C}_{T-1}%
\end{array}
\right)  \mathbb{E}\left[  p_{T}\left(  \Delta W_{T-1}\right)  |\mathcal{F}%
_{T-1}\right]  +\left(
\begin{array}
[c]{c}%
D_{T-1}\\
\overline{D}_{T-1}%
\end{array}
\right)  \right]  ,\\
h_{T-1} & = & \left(  0,P_{T}\right)  \left[  \Gamma_{T-1}\left(
P_{T}\right)  \right]  ^{-1}\left[  \left(
\begin{array}
[c]{c}%
B_{T-1}\\
\overline{B}_{T-1}%
\end{array}
\right)  \mathbb{E}\left[  p_{T}|\mathcal{F}_{T-1}\right]  +\left(
\begin{array}
[c]{c}%
D_{T-1}\\
\overline{D}_{T-1}%
\end{array}
\right)  \right] \\
&  & +\left[  I+\left(  0,P_{T}\right)  \left[  \Gamma_{T-1}\left(
P_{T}\right)  \right]  ^{-1}\left(
\begin{array}
[c]{c}%
C_{T-1}\\
\overline{C}_{T-1}%
\end{array}
\right)  \right]  \mathbb{E}\left[  p_{T}\left(  \Delta W_{T-1}\right)
|\mathcal{F}_{T-1}\right]  .
\end{array}
\right.
\]
Let
\[
\lambda_{T-1}=-\widehat{A}_{T-1}X_{T-1}+\left(  I-\widehat{B}_{T-1}\right)
Y_{T-1}-\widehat{C}_{T-1}Z_{T-1}-\widehat{D}_{T-1}.
\]
By (\ref{c3_eq_w_fbsde_linear_2_yz}),\ we have%
\[
\lambda_{T-1}=P_{T-1}X_{T-1}+p_{T-1}%
\]
where%
\begin{align*}
P_{T-1}  &  =-\widehat{A}_{T-1}+\left(  I-\widehat{B}_{T-1}\right)
G_{T-1}-\widehat{C}_{T-1}H_{T-1},\\
p_{T-1}  &  =\left(  I-\widehat{B}_{T-1}\right)  g_{T-1}-\widehat{C}%
_{T-1}h_{T-1}-\widehat{D}_{T-1}.
\end{align*}

Repeating the above procedure, we can obtain the result by backward induction.
This completes the proof.
\end{proof}

\begin{remark}
By Theorem \ref{solution theorem for linear fbsde_4}, the solvability of
FBS$\Delta$E (\ref{fbsde_linear_1d_4}) only depends on the coefficients of the
homogeneous terms.
\end{remark}

The following corollary will be used in the proof of the solvability of the
nonlinear FBS$\Delta$E (\ref{nonlinear_fbsde_c_method_0}).

\begin{corollary}
\label{c3_cor_w_fbsde_linear_2}For any $D,\overline{D}\in\mathcal{M}%
^{2}\left(  0,T-1;\mathbb{R}^{m}\right)  $, $\widehat{D}\in\mathcal{M}%
^{2}\left(  1,T;\mathbb{R}^{n}\right)  $, $g\in L^{2}\left(  \mathcal{F}%
_{T};\mathbb{R}^{n}\right)  $, the following linear FBS$\Delta$E%
\[
\left\{
\begin{array}
[c]{rcl}%
\Delta X_{t} & = & -\beta_{2}G^{\ast}Y_{t}+D_{t}+\left(  -\beta_{2}G^{\ast
}Z_{t}+\overline{D}_{t}\right)  \Delta W_{t},\\
\Delta Y_{t} & = & -\beta_{1}GX_{t+1}+\widehat{D}_{t+1}+Z_{t}\Delta
W_{t}+\Delta N_{t},\\
X_{0} & = & x_{0},\\
Y_{T} & = & GX_{T}+g.
\end{array}
\right.
\]
has a unique solution $\left(  X,Y,Z,N\right)  \in\mathcal{M}^{2}\left(
0,T;\mathbb{R}^{m}\right)  \times\mathcal{M}^{2}\left(  0,T;\mathbb{R}%
^{n}\right)  \times\mathcal{M}^{2}\left(  0,T-1;\mathbb{R}^{n}\right)
\times\mathcal{M}^{2}\left(  0,T;\mathbb{R}^{n}\right)  $.
\end{corollary}

\begin{proof}
For this special case of (\ref{fbsde_linear_1d_4}), the coefficients are%
\begin{align*}
A_{t}  &  =\overline{A}_{t}=\overline{B}_{t}=\widehat{B}_{t}=C_{t}%
=\widehat{C}_{t}=0,\\
B_{t}  &  =\overline{C}_{t}=-\beta_{2}G^{\ast}\\
\widehat{A}_{t}  &  =-\beta_{1}G\\
G  &  =G.
\end{align*}
Then we have%
\[
\Gamma_{t}\left(  P_{t+1}\right)  =I-\left(
\begin{array}
[c]{cc}%
B_{t}P_{t+1} & C_{t}P_{t+1}\\
\overline{B}_{t}P_{t+1} & \overline{C}_{t}P_{t+1}%
\end{array}
\right)  =\left(
\begin{array}
[c]{cc}%
I+\beta_{2}G^{\ast}P_{t+1} & 0\\
0 & I+\beta_{2}G^{\ast}P_{t+1}%
\end{array}
\right)
\]
and%
\[%
\begin{array}
[c]{rcl}%
P_{t} & = & -\widehat{A}_{t}+\left(  \left(  I-\widehat{B}_{t}\right)
P_{t+1},-\widehat{C}_{t}P_{t+1}\right)  \left[  \Gamma_{t}\left(
P_{t+1}\right)  \right]  ^{-1}\left(
\begin{array}
[c]{c}%
I+A_{t}\\
\overline{A}_{t}%
\end{array}
\right) \\
& = & \beta_{1}G+P_{t+1}\left(  I+\beta_{2}G^{\ast}P_{t+1}\right)  ^{-1}.
\end{array}
\]
Since $P_{T}=\left(  1+\beta_{1}\right)  G$, it is easy to check that $\forall
t\in\left\{  1,...,T\right\}  $,\ $\Gamma_{t-1}\left(  P_{t}\right)  $ is
invertible. Thus, the above FBS$\Delta$E have a unique solution. This
completes the proof.
\end{proof}

\section{Solvability of nonlinear FBS$\Delta$Es}

In this section we study the solvability of\ nonlinear FBS$\Delta
$E\ (\ref{nonlinear_fbsde_c_method_0}).

\begin{theorem}
\label{Thm-nonlinear}Under Assumption \ref{assu-nonlinear}, FBS$\Delta
$E\thinspace(\ref{nonlinear_fbsde_c_method_0}) has a unique adapted solution
$\left(  X,Y,Z,N\right)  \in\mathcal{M}^{2}\left(  0,T;\mathbb{R}^{m}\right)
\times\mathcal{M}^{2}\left(  0,T;\mathbb{R}^{n}\right)  \times\mathcal{M}%
^{2}\left(  0,T-1;\mathbb{R}^{n}\right)  \times\mathcal{M}^{2}\left(
0,T;\mathbb{R}^{n}\right)  $.
\end{theorem}

We first give the proof of the uniqueness.

\begin{proof}
Suppose $U=\left(  X,Y,Z,N\right)  $ and $U^{^{\prime}}=\left(  X^{^{\prime}%
},Y^{^{\prime}},Z^{^{\prime}},N^{^{\prime}}\right)  $ are two solutions of
FBS$\Delta$E\thinspace(\ref{nonlinear_fbsde_c_method_0}). Define%
\[
\left(  \widehat{X},\widehat{Y},\widehat{Z},\widehat{N}\right)  =\left(
X-X^{^{\prime}},Y-Y^{^{\prime}},Z-Z^{^{\prime}},N-N^{^{\prime}}\right)  ,
\]
and%
\begin{align*}
\widehat{b}\left(  t\right)   &  =b\left(  t,X_{t},Y_{t},Z_{t}\right)
-b\left(  t,X_{t}^{\prime},Y_{t}^{\prime},Z_{t}^{^{\prime}}\right)  ,\\
\widehat{\sigma}\left(  t\right)   &  =\sigma\left(  t,X_{t},Y_{t}%
,Z_{t}\right)  -\sigma\left(  t,X_{t}^{\prime},Y_{t}^{\prime},Z_{t}^{\prime
}\right)  ,\\
\widehat{f}\left(  t\right)   &  =f\left(  t,X_{t},Y_{t},Z_{t}\right)
-f\left(  t,X_{t}^{\prime},Y_{t}^{^{\prime}},Z_{t}^{^{\prime}}\right)  .
\end{align*}

For $t\in\left\{  0,1,...,T-1\right\}  $, we have%
\[%
\begin{array}
[c]{cl}
& \Delta\left\langle G\widehat{X}_{t},\widehat{Y}_{t}\right\rangle \\
= & \left\langle G\widehat{X}_{t+1},\Delta\widehat{Y}_{t}\right\rangle
+\left\langle G\Delta\widehat{X}_{t},\widehat{Y}_{t}\right\rangle \\
= & \left\langle G\widehat{X}_{t+1},-\widehat{f}\left(  t+1\right)
\right\rangle +\left\langle G\widehat{X}_{t}+G\Delta\widehat{X}_{t}%
,\widehat{Z}_{t}\Delta W_{t}\right\rangle +\left\langle G\widehat{X}%
_{t}+G\Delta\widehat{X}_{t},\Delta\widehat{N}_{t}\right\rangle \\
& +\left\langle G\widehat{b}\left(  t\right)  ,\widehat{Y}_{t}\right\rangle
+\left\langle G\widehat{\sigma}\left(  t\right)  \Delta W_{t},\widehat{Y}%
_{t}\right\rangle \\
= & \left\langle G\widehat{X}_{t+1},-\widehat{f}\left(  t+1\right)
\right\rangle +\left\langle G\widehat{\sigma}\left(  t\right)  \Delta
W_{t},\widehat{Z}_{t}\Delta W_{t}\right\rangle +\left\langle G\widehat{b}%
\left(  t\right)  ,\widehat{Y}_{t}\right\rangle +\Phi_{t}%
\end{array}
\]
where%
\[
\Phi_{t}=\left\langle G\left(  \widehat{X}_{t}+\widehat{b}\left(  t\right)
\right)  ,\widehat{Z}_{t}\Delta W_{t}\right\rangle +\left\langle G\left(
\widehat{X}_{t}+\Delta\widehat{X}_{t}\right)  ,\Delta\widehat{N}%
_{t}\right\rangle +\left\langle G\widehat{\sigma}\left(  t\right)  \Delta
W_{t},\widehat{Y}_{t}\right\rangle .
\]
Since $W$, $N$, $N^{^{\prime}}$ are martingale process and $N$, $N^{^{\prime}%
}$ are strongly orthogonal to $W$, we have $\mathbb{E}\left[  \Phi
_{t}|\mathcal{F}_{t}\right]  =0$. Notice that%
\[
\mathbb{E}\left\langle G\widehat{\sigma}\left(  t\right)  \Delta
W_{t},\widehat{Z}_{t}\Delta W_{t}\right\rangle =\mathbb{E}\left[  \left\langle
G\widehat{\sigma}\left(  t\right)  ,\widehat{Z}_{t}\right\rangle
\mathbb{E}\left[  \left(  \Delta W_{t}\right)  ^{2}|\mathcal{F}_{t}\right]
\right]  =\mathbb{E}\left\langle G\widehat{\sigma}\left(  t\right)
,\widehat{Z}_{t}\right\rangle .
\]
Then,%
\[%
\begin{array}
[c]{cl}
& \mathbb{E}\left\langle G\left(  X_{T}-X_{T}^{\prime}\right)  ,h\left(
X_{T}\right)  -h\left(  X_{T}^{\prime}\right)  \right\rangle \\
= & \mathbb{E}\left\langle G\widehat{X}_{T},\widehat{Y}_{T}\right\rangle \\
= & \mathbb{E}\sum\limits_{t=0}^{T-1}\Delta\left\langle G\widehat{X}%
_{t},\widehat{Y}_{t}\right\rangle \\
= & \mathbb{E}\left[  \sum\limits_{t=0}^{T-1}\left\langle G\widehat{X}%
_{t+1},-\widehat{f}\left(  t+1\right)  \right\rangle +\sum\limits_{t=0}%
^{T-1}\left\langle G\widehat{\sigma}\left(  t\right)  ,\widehat{Z}%
_{t}\right\rangle +\sum\limits_{t=0}^{T-1}\left\langle G\widehat{b}\left(
t\right)  ,\widehat{Y}_{t}\right\rangle \right] \\
= & \mathbb{E}\left[  \sum\limits_{t=1}^{T-1}\left\langle A\left(
t,\Lambda\right)  -A\left(  t,\Lambda^{^{\prime}}\right)  ,\Lambda
-\Lambda^{^{\prime}}\right\rangle +\left\langle G\widehat{b}\left(  0\right)
,\widehat{Y}_{0}\right\rangle +\left\langle G\widehat{\sigma}\left(  0\right)
,\widehat{Z}_{0}\right\rangle +\left\langle G\widehat{X}_{T},-\widehat{f}%
\left(  T\right)  \right\rangle \right]  .
\end{array}
\]
Due to the monotone condition, we obtain%
\[%
\begin{array}
[c]{ccl}%
0 & \leq & \mathbb{E}\left\langle G\left(  X_{T}-X_{T}^{\prime}\right)
,h\left(  X_{T}\right)  -h\left(  X_{T}^{\prime}\right)  \right\rangle \\
& = & \mathbb{E}\left[  \sum\limits_{t=1}^{T-1}\left\langle A\left(
t,\Lambda\right)  -A\left(  t,\Lambda^{^{\prime}}\right)  ,\Lambda
-\Lambda^{^{\prime}}\right\rangle +\left\langle \widehat{b}\left(  0\right)
,\widehat{Y}_{0}\right\rangle +\left\langle \widehat{\sigma}\left(  0\right)
,\widehat{Z}_{0}\right\rangle +\left\langle \widehat{X}_{T},-\widehat{f}%
\left(  T\right)  \right\rangle \right] \\
& \leq & -\beta_{1}\mathbb{E}\sum\limits_{t=0}^{T}\left\vert G\left(
X_{t}-X_{t}^{^{\prime}}\right)  \right\vert ^{2}-\beta_{2}\mathbb{E}%
\sum\limits_{t=0}^{T-1}\left\vert G^{\ast}\left(  Y_{t}-Y_{t}^{^{\prime}%
}\right)  \right\vert ^{2}-\beta_{2}\mathbb{E}\sum\limits_{t=0}^{T-1}%
\left\vert G^{\ast}\left(  Z_{t}-Z_{t}^{^{\prime}}\right)  \right\vert ^{2}.
\end{array}
\]

Consider the case where $m<n$. In this case $\beta_{1}>0$, then we obtain
$X_{t}-X_{t}^{^{\prime}}=0$, $P-a.s.$ for $t\in\left\{  0,...,T\right\}  $.
Thus from the uniqueness of BS$\Delta$E, it follows that $Y=Y^{^{\prime}}$,
$Z=Z^{^{\prime}}$ and $N=N^{^{\prime}}$.

Consider the case where $m>n$. In this case $\beta_{2}>0$, then we obtain
$Y_{t}-Y_{t}^{^{\prime}}=0$, $Z_{t}-Z_{t}^{^{\prime}}=0$, $P-a.s.$ for
$t\in\left\{  0,...,T-1\right\}  $. Thus from the uniqueness of the forward
equation, we have $X=X^{^{\prime}}$. In particular, $h\left(  X_{T}\right)
=h\left(  X_{T}^{\prime}\right)  $. Thus we have $Y_{T}=Y_{T}^{^{\prime}}$,
$P-a.s.$.\ Also it is easy to check that $\mathbb{E}\sum_{t=0}^{T}\left\vert
N_{t}-N_{t}^{^{\prime}}\right\vert ^{2}=0$.\ Hence $U=U^{^{\prime}}$.

Similarly to the above two cases, the result can be obtained easily in the
case $m=n$.\ This completes the proof of the uniqueness.
\end{proof}

In order to prove the existence part of Theorem \ref{Thm-nonlinear}, we
introduce the following family of FBS$\Delta$Es parameterized by $\alpha
\in\left[  0,1\right]  $:%
\begin{equation}
\left\{
\begin{array}
[c]{rcl}%
\Delta X_{t} & = & b^{\alpha}\left(  t,X_{t},Y_{t},Z_{t}\right)  +b_{0}\left(
t\right)  +\left[  \sigma^{\alpha}\left(  t,X_{t},Y_{t},Z_{t}\right)
+\sigma_{0}\left(  t\right)  \right]  \Delta W_{t},\\
\Delta Y_{T-1} & = & -g^{\alpha}\left(  X_{T},Y_{T}\right)  -f_{0}\left(
T\right)  +Z_{T-1}\Delta W_{T-1}+\Delta N_{T-1},\\
\Delta Y_{t} & = & -f^{\alpha}\left(  t+1,X_{t+1},Y_{t+1},Z_{t+1}\right)
-f_{0}\left(  t+1\right)  +Z_{t}\Delta W_{t}+\Delta N_{t},\\
X_{0} & = & x_{0},\\
Y_{T} & = & h^{\alpha}\left(  X_{T}\right)  +h_{0},
\end{array}
\right.  \label{nonlinear_fbsde_c_method_1}%
\end{equation}
where%
\begin{align*}
b^{\alpha}\left(  t,x,y,z\right)   &  =\alpha b\left(  t,x,y,z\right)
+\left(  1-\alpha\right)  \beta_{2}\left(  -G^{\ast}y\right)  ,\\
\sigma^{\alpha}\left(  t,x,y,z\right)   &  =\alpha\sigma\left(
t,x,y,z\right)  +\left(  1-\alpha\right)  \beta_{2}\left(  -G^{\ast}z\right)
,\\
f^{\alpha}\left(  t,x,y,z\right)   &  =\alpha f\left(  t,x,y,z\right)
+\left(  1-\alpha\right)  \beta_{1}Gx,\\
h^{a}\left(  x\right)   &  =\alpha h\left(  x\right)  +\left(  1-\alpha
\right)  Gx.
\end{align*}
Clearly, when $\alpha=1$, FBS$\Delta$E (\ref{nonlinear_fbsde_c_method_1})
becomes FBS$\Delta$E (\ref{nonlinear_fbsde_c_method_0}). On the other hand,
when $\alpha=0$, FBS$\Delta$E (\ref{nonlinear_fbsde_c_method_1}) becomes the
following linear equation:%
\begin{equation}
\left\{
\begin{array}
[c]{rcl}%
\Delta X_{t} & = & -\beta_{2}G^{\ast}Y_{t}+b_{0}\left(  t\right)  +\left[
-\beta_{2}G^{\ast}Z_{t}+\sigma_{0}\left(  t\right)  \right]  \Delta W_{t},\\
\Delta Y_{t} & = & -\beta_{1}GX_{t+1}-f_{0}\left(  t+1\right)  +Z_{t}\Delta
W_{t}+\Delta N_{t},\\
X_{0} & = & x_{0},\\
Y_{T} & = & GX_{T}+h_{0}.
\end{array}
\right.  \label{nonlinear_fbsde_c_method_2}%
\end{equation}

By Corollary \ref{c3_cor_w_fbsde_linear_2}, we know that for any $b_{0},$
$\sigma_{0}\in\mathcal{M}^{2}\left(  0,T-1;\mathbb{R}^{n}\right)  $, $f_{0}%
\in\mathcal{M}^{2}\left(  1,T;\mathbb{R}^{n}\right)  $ and $h_{0}\in
L^{2}\left(  \mathcal{F}_{T};\mathbb{R}^{n}\right)  $, the linear FBS$\Delta
$E\ (\ref{nonlinear_fbsde_c_method_2}) has a unique solution.

To complete the proof, we need the following lemma.

\begin{lemma}
\label{nonlinear_fbsde_c_method_lemma}Assume that there exists an $\alpha
_{0}\in\left[  0,1\right)  $ such that for any $b_{0},$ $\sigma_{0}%
\in\mathcal{M}^{2}\left(  0,T-1;\mathbb{R}^{n}\right)  ,$ $f_{0}\in
\mathcal{M}^{2}\left(  1,T;\mathbb{R}^{n}\right)  $ and $h_{0}\in L^{2}\left(
\mathcal{F}_{T};\mathbb{R}^{n}\right)  $, FBS$\Delta$E
(\ref{nonlinear_fbsde_c_method_1}) has a unique solution. Then there exists
$\delta_{0}\in\left(  0,1\right)  $, which depends on $\beta_{1},\beta_{2}$
and $T$, such that for any $\alpha\in\left[  \alpha_{0},\alpha_{0}+\delta
_{0}\right]  $, any $b_{0},$ $\sigma_{0}\in\mathcal{M}^{2}\left(
0,T-1;\mathbb{R}^{n}\right)  ,$ $f_{0}\in\mathcal{M}^{2}\left(  1,T;\mathbb{R}%
^{n}\right)  $ and $h_{0}\in L^{2}\left(  \mathcal{F}_{T};\mathbb{R}%
^{n}\right)  $, FBS$\Delta$E (\ref{nonlinear_fbsde_c_method_1}) has a unique solution.
\end{lemma}

\begin{proof}
Observe that%
\begin{align*}
b^{\alpha_{0}+\delta}\left(  t,x,y,z\right)   &  =b^{\alpha_{0}}\left(
t,x,y,z\right)  +\delta\left(  \beta_{2}G^{\ast}y+b\left(  t,x,y,z\right)
\right)  ,\\
\sigma^{\alpha_{0}+\delta}\left(  t,x,y,z\right)   &  =\sigma^{\alpha_{0}%
}\left(  t,x,y,z\right)  +\delta\left(  \beta_{2}G^{\ast}z+\sigma\left(
t,x,y,z\right)  \right)  ,\\
f^{\alpha_{0}+\delta}\left(  t,x,y,z\right)   &  =f^{\alpha_{0}}\left(
t,x,y,z\right)  +\delta\left(  -\beta_{1}Gx+f\left(  t,x,y,z\right)  \right)
,\\
h^{\alpha_{0}+\delta}\left(  x\right)   &  =h^{\alpha_{0}}\left(  x\right)
+\delta\left(  -Gx+h\left(  x\right)  \right)  .
\end{align*}
According to the assumption, for each triple $u=\left(  x_{s},y_{s}%
,z_{s}\right)  \in\mathcal{M}^{2}\left(  0,T;\mathbb{R}^{m}\right)
\times\mathcal{M}^{2}\left(  0,T;\mathbb{R}^{n}\right)  \times\mathcal{M}%
^{2}\left(  0,T-1;\mathbb{R}^{n}\right)  $,\ there exists a unique triple
$U=\left(  X_{s},Y_{s},Z_{s}\right)  \in\mathcal{M}^{2}\left(  0,T;\mathbb{R}%
^{m}\right)  \times\mathcal{M}^{2}\left(  0,T;\mathbb{R}^{n}\right)
\times\mathcal{M}^{2}\left(  0,T-1;\mathbb{R}^{n}\right)  $ satisfying the
following equation
\begin{equation}
\left\{
\begin{array}
[c]{rcl}%
\Delta X_{t} & = & b^{\alpha_{0}}\left(  t,U_{t}\right)  +\delta\left(
\beta_{2}G^{\ast}y_{t}+b\left(  t,u_{t}\right)  \right)  +b_{0}\left(
t\right)  \\
&  & +\left[  \sigma^{\alpha_{0}}\left(  t,U_{t}\right)  +\delta\left(
\beta_{2}G^{\ast}z_{t}+\sigma\left(  t,u_{t}\right)  \right)  +\sigma
_{0}\left(  t\right)  \right]  \Delta W_{t},\\
\Delta Y_{t} & = & -f^{\alpha_{0}}\left(  t+1,U_{t+1}\right)  -\delta\left(
-\beta_{1}Gx_{t+1}+f\left(  t+1,u_{t+1}\right)  \right)  \\
&  & -f_{0}\left(  t+1\right)  +Z_{t}\Delta W_{t}+\Delta N_{t},\\
X_{0} & = & x_{0},\\
Y_{T} & = & h^{\alpha_{0}}\left(  X_{T}\right)  +\delta\left(  -Gx_{T}%
+h\left(  x_{T}\right)  \right)  +h_{0}.
\end{array}
\right.  \label{nonlinear_fbsde_c_method_eq2}%
\end{equation}

We are going to prove that the mapping defined by
\begin{align*}
I_{\alpha_{0}+\delta}\left(  u\right)   &  =U:\mathcal{M}^{2}\left(
0,T;\mathbb{R}^{m}\right)  \times\mathcal{M}^{2}\left(  0,T;\mathbb{R}%
^{n}\right)  \times\mathcal{M}^{2}\left(  0,T-1;\mathbb{R}^{n}\right) \\
&  \rightarrow\mathcal{M}^{2}\left(  0,T;\mathbb{R}^{m}\right)  \times
\mathcal{M}^{2}\left(  0,T;\mathbb{R}^{n}\right)  \times\mathcal{M}^{2}\left(
0,T-1;\mathbb{R}^{n}\right)
\end{align*}
is a contraction mapping.

Let $u^{^{\prime}}=\left(  x_{s}^{^{\prime}},y_{s}^{^{\prime}},z_{s}%
^{^{\prime}}\right)  \in\mathcal{M}^{2}\left(  0,T;\mathbb{R}^{m}\right)
\times\mathcal{M}^{2}\left(  0,T;\mathbb{R}^{n}\right)  \times\mathcal{M}%
^{2}\left(  0,T-1;\mathbb{R}^{n}\right)  $, and let $U^{^{\prime}}=\left(
X_{s}^{^{\prime}},Y_{s}^{^{\prime}},Z_{s}^{^{\prime}}\right)  =I_{\alpha
_{0}+\delta}\left(  u^{^{\prime}}\right)  $. We set $\widehat{u}%
=u-u^{^{\prime}}=\left(  \widehat{x},\widehat{y},\widehat{z}\right)  $,
$\widehat{U}=U-U^{^{\prime}}=\left(  \widehat{X},\widehat{Y},\widehat{Z}%
\right)  .$ Then we have%
\[%
\begin{array}
[c]{cl}
& \mathbb{E}\left[  \left\langle G\widehat{X}_{T},h^{\alpha_{0}}\left(
X_{T}\right)  -h^{\alpha_{0}}\left(  X_{T}^{^{\prime}}\right)  \right\rangle
+\left\langle G\widehat{X}_{T},\delta\left(  -G\left(  x_{T}-x_{T}^{\prime
}\right)  +h\left(  x_{T}\right)  -h\left(  x_{T}^{\prime}\right)  \right)
\right\rangle \right] \\
= & \mathbb{E}\left\langle G\widehat{X}_{T},\widehat{Y}_{T}\right\rangle \\
= & \mathbb{E}\left[  \sum\limits_{t=1}^{T-1}\left\langle A^{\alpha_{0}%
}\left(  t,U_{t}\right)  -A^{\alpha_{0}}\left(  t,U_{t}^{^{\prime}}\right)
,\widehat{U}_{t}\right\rangle +\left\langle Gb^{\alpha_{0}}\left(
0,U_{0}\right)  -Gb^{\alpha_{0}}\left(  0,U_{0}^{^{\prime}}\right)
,\widehat{Y}_{0}\right\rangle \right. \\
& +\left\langle G\sigma^{\alpha_{0}}\left(  0,U_{0}\right)  -G\sigma
^{\alpha_{0}}\left(  0,U_{0}^{^{\prime}}\right)  ,\widehat{Z}_{0}\right\rangle
-\left\langle G^{\ast}f^{\alpha_{0}}\left(  T,X_{T},Y_{T}\right)  -G^{\ast
}f^{\alpha_{0}}\left(  T,X_{T}^{\prime},Y_{T}^{\prime}\right)  ,\widehat{X}%
_{T}\right\rangle \\
& +\sum\limits_{t=1}^{T-1}\delta\left\langle A\left(  t,u_{t}\right)
-A\left(  t,u_{t}^{^{\prime}}\right)  +\widehat{u}_{t},\widehat{U}%
_{t}\right\rangle +\delta\left\langle Gb\left(  0,u_{0}\right)  -Gb\left(
0,u_{0}^{\prime}\right)  ,\widehat{Y}_{0}\right\rangle \\
& +\delta\left\langle G\sigma\left(  0,u_{0}\right)  -G\sigma\left(
0,u_{0}^{\prime}\right)  ,\widehat{Z}_{0}\right\rangle -\delta\left\langle
G^{\ast}f\left(  T,x_{T},y_{T}\right)  -G^{\ast}f\left(  T,x_{T}^{\prime
},y_{T}^{\prime}\right)  ,\widehat{X}_{T}\right\rangle \\
& +\sum\limits_{t=1}^{T-1}\delta\left(  \beta_{1}\left\langle G\widehat{x}%
_{t},G\widehat{X}_{t}\right\rangle +\beta_{2}\left\langle G^{\ast}%
\widehat{y}_{t},G^{\ast}\widehat{Y}_{t}\right\rangle +\beta_{2}\left\langle
G^{\ast}\widehat{z}_{t},G^{\ast}\widehat{Z}_{t}\right\rangle \right) \\
& \left.  +\delta\beta_{2}\left\langle G^{\ast}\widehat{y}_{0},G^{\ast
}\widehat{Y}_{0}\right\rangle +\delta\beta_{2}\left\langle G^{\ast}%
\widehat{z}_{0},G^{\ast}\widehat{Z}_{0}\right\rangle +\delta\beta
_{1}\left\langle G\widehat{x}_{T},G\widehat{X}_{T}\right\rangle \right]
\end{array}
\]
where $A^{\alpha_{0}}\left(  t,u\right)  =\left(
\begin{array}
[c]{c}%
-G^{\ast}f^{\alpha_{0}}\\
Gb^{\alpha_{0}}\\
G\sigma^{\alpha_{0}}%
\end{array}
\right)  \left(  t,u\right)  .$

By Assumption \ref{assu-nonlinear}, we obtain%
\begin{equation}%
\begin{array}
[c]{cl}
& \mathbb{E}\left[  \sum_{t=0}^{T-1}\left(  \beta_{1}\left\vert G\widehat{X}%
_{t}\right\vert ^{2}+\beta_{2}\left\vert G^{\ast}\widehat{Y}_{t}\right\vert
^{2}+\beta_{2}\left\vert G^{\ast}\widehat{Z}_{t}\right\vert ^{2}\right)
+\left(  \beta_{1}+1-\alpha_{0}\right)  \left\vert G\widehat{X}_{T}\right\vert
^{2}\right] \\
\leq & \delta K_{1}\mathbb{E}\left[  \sum_{t=0}^{T-1}\left(  \left\vert
\widehat{U}_{t}\right\vert ^{2}+\left\vert \widehat{u}_{t}\right\vert
^{2}\right)  +\left\vert \widehat{X}_{T}\right\vert ^{2}+\left\vert
\widehat{x}_{T}\right\vert ^{2}+\left\vert \widehat{y}_{T}\right\vert
^{2}\right]  ,
\end{array}
\label{nonlinear_fbsde_c_method_eq2_1}%
\end{equation}
where the constant $K_{1}$ depends on the constants $\beta_{1}$, $\beta_{2}$,
$G$ and $c.$

On the other hand, applying Corollary \ref{bsde_result_3} to the BS$\Delta$E
in (\ref{nonlinear_fbsde_c_method_eq2}), we can obtain%
\begin{equation}
\sum_{t=0}^{T-1}\mathbb{E}\left[  \left\vert \widehat{Y}_{t}\right\vert
^{2}+\left\Vert \widehat{Z}_{t}\right\Vert ^{2}+\left\vert \Delta
\widehat{N}_{t}\right\vert ^{2}\right]  \leq K_{2}\mathbb{E}\left[  \sum
_{t=1}^{T}\left\vert \widehat{X}_{t}\right\vert ^{2}+\delta\left(  \left\vert
\widehat{x}_{T}\right\vert ^{2}+\left\vert \widehat{y}_{T}\right\vert
^{2}+\sum_{t=1}^{T-1}\left\vert \widehat{u}_{t}\right\vert ^{2}\right)
\right]  ,\label{nonlinear_fbsde_c_method_eq3}%
\end{equation}
where the constant $K_{2}$ depends on the constants $\beta_{1}$, $G$, $c$ and
$T.$ According to the terminal condition,\ we have%
\begin{equation}
\mathbb{E}\left[  \left\vert \widehat{Y}_{T}\right\vert ^{2}\right]  \leq
K_{2}\mathbb{E}\left[  \left\vert \widehat{X}_{T}\right\vert ^{2}%
+\delta\left\vert \widehat{x}_{T}\right\vert ^{2}\right]
.\label{nonlinear_fbsde_c_method_eq4}%
\end{equation}

Besides, by applying the induction method to the S$\Delta$E in
(\ref{nonlinear_fbsde_c_method_eq2}), we can obtain%
\begin{equation}
\sum_{t=0}^{T}\mathbb{E}\left[  \left\vert \widehat{X}_{t}\right\vert
^{2}\right]  \leq K_{3}\mathbb{E}\left[  \sum_{t=1}^{T-1}\left(  \left\vert
\widehat{Y}_{t}\right\vert ^{2}+\left\vert \widehat{Z}_{t}\right\vert
^{2}\right)  +\delta\left(  \sum_{t=1}^{T-1}\left\vert \widehat{u}%
_{t}\right\vert ^{2}\right)  \right]  \label{nonlinear_fbsde_c_method_eq5}%
\end{equation}
where the constant $K_{3}$ depends on the constants $\beta_{2}$, $G$,\ $c$ and
$T.$

In the case $\beta_{1}>0$ (resp. $\beta_{2}>0$), combining equation
(\ref{nonlinear_fbsde_c_method_eq2_1}), (\ref{nonlinear_fbsde_c_method_eq3}),
(\ref{nonlinear_fbsde_c_method_eq4}) (resp. equation
(\ref{nonlinear_fbsde_c_method_eq2_1}), (\ref{nonlinear_fbsde_c_method_eq4}),
(\ref{nonlinear_fbsde_c_method_eq5})), we always have%
\[
\mathbb{E}\left[  \sum_{t=0}^{T-1}\left\vert \widehat{U}_{t}\right\vert
^{2}+\left\vert \widehat{X}_{T}\right\vert ^{2}+\left\vert \widehat{Y}%
_{T}\right\vert ^{2}\right]  \leq\delta K\mathbb{E}\left[  \sum_{t=0}%
^{T-1}\left\vert \widehat{u}_{t}\right\vert ^{2}+\left\vert \widehat{x}%
_{T}\right\vert ^{2}+\left\vert \widehat{y}_{T}\right\vert ^{2}\right]  ,
\]
where the constant $K$ depends only on\ $\beta_{1}$, $\beta_{2}$, $G$,\ $c$
and $T.$ We now choose $\delta_{0}=\frac{1}{2K}$. It is clear that, for each
fixed $\delta\in\left[  0,\delta_{0}\right]  $, the mapping $I_{\alpha
_{0}+\delta}$ is a contraction\ mapping in the sense that%
\[
\mathbb{E}\left[  \sum_{t=0}^{T-1}\left\vert \widehat{U}_{t}\right\vert
^{2}+\left\vert \widehat{X}_{T}\right\vert ^{2}+\left\vert \widehat{Y}%
_{T}\right\vert ^{2}\right]  \leq\frac{1}{2}\mathbb{E}\left[  \sum_{t=0}%
^{T-1}\left\vert \widehat{u}_{t}\right\vert ^{2}+\left\vert \widehat{x}%
_{T}\right\vert ^{2}+\left\vert \widehat{y}_{T}\right\vert ^{2}\right]  .
\]

It follows that this mapping has a unique fixed point $\left(  X^{\alpha
_{0}+\delta},Y^{\alpha_{0}+\delta},Z^{\alpha_{0}+\delta}\right)  $ which
satisfies equation (\ref{nonlinear_fbsde_c_method_1}) for $\alpha=\alpha
_{0}+\delta$. Besides, $\left\{  N_{t}^{\alpha_{0}+\delta}\right\}  _{t=0}%
^{T}$ can be obtained such that $\left(  X,Y,Z,N\right)  $ is the solution to
FBS$\Delta$E (\ref{nonlinear_fbsde_c_method_1}) with $\alpha=\alpha_{0}%
+\delta$. This completes the proof of this lemma.
\end{proof}

Now we complete the proof of the existence part of Theorem \ref{Thm-nonlinear}.

\begin{proof}
When $\alpha=0$, for any $b_{0},$ $\sigma_{0}\in\mathcal{M}^{2}\left(
0,T-1;\mathbb{R}^{m}\right)  ,$ $f_{0}\in\mathcal{M}^{2}\left(  1,T;\mathbb{R}%
^{n}\right)  $ and $h_{0}\in L^{2}\left(  \mathcal{F}_{T};\mathbb{R}%
^{n}\right)  $, there exists a unique solution to FBS$\Delta$%
E\ (\ref{nonlinear_fbsde_c_method_1}). According to Lemma
\ref{nonlinear_fbsde_c_method_lemma}, there exists a constant $\delta_{0}$
which only depends on $\beta_{1}$, $\beta_{2}$, $G$,\ $c$ and $T$, such that
for any $b_{0},$ $\sigma_{0}\in\mathcal{M}^{2}\left(  0,T-1;\mathbb{R}%
^{m}\right)  ,$ $f_{0}\in\mathcal{M}^{2}\left(  1,T;\mathbb{R}^{n}\right)  $
and $h_{0}\in L^{2}\left(  \mathcal{F}_{T};\mathbb{R}^{n}\right)  $,
FBS$\Delta$E (\ref{nonlinear_fbsde_c_method_1}) has a unique solution
for\ $\alpha\in\left[  0,\delta_{0}\right]  ,$ $\left[  \delta_{0},2\delta
_{0}\right]  ,...$ It turns out that FBS$\Delta$E
(\ref{nonlinear_fbsde_c_method_1}) has a unique solution when $\alpha=1$.
Taking $b_{0}\left(  \cdot\right)  =\sigma_{0}\left(  \cdot\right)
=f_{0}\left(  \cdot\right)  =0$ and $h_{0}=0$, we deduce that the solution of
FBS$\Delta$E (\ref{nonlinear_fbsde_c_method_0}) exists. This completes the
proof of the existence.
\end{proof}

\section{Appendix}

In this appendix, we give the proofs of Theorem \ref{bsde_result_l} and
\ref{bsde_result_2}.

\begin{proof}
[Proof of Theorem \ref{bsde_result_l}]We first prove the existence and
uniqueness of $\left(  Y_{T-1},Z_{T-1},\Delta N_{T-1}\right)  $. Due to
Assumption (\ref{generator_assumption}) and $\eta\in L^{2}\left(
\mathcal{F}_{T};\mathbb{R}^{n}\right)  $, we get $f\left(  T,\eta\right)  \in
L^{2}\left(  \mathcal{F}_{T};\mathbb{R}^{n}\right)  $. Here we omit the
variable $Z$ since $f$ is independent of $Z$ at time $T$. Then we have
$\mathbb{E}\left[  \left\vert \mathbb{E}\left[  \eta+f\left(  T,\eta\right)
|\mathcal{F}_{T-1}\right]  \right\vert ^{2}\right]  <\infty.$ Hence,
$\eta+f\left(  T,\eta\right)  -\mathbb{E}\left[  \eta+f\left(  T,\eta\right)
|\mathcal{F}_{T-1}\right]  $ is a square integrable martingale difference. So
it admits the Galtchouk-Kunita-Watanabe decomposition, which implies that
there exists $Z_{T-1}\in\mathcal{F}_{T-1}$, $Z_{T-1}\Delta W_{T-1}\in
L^{2}\left(  \mathcal{F}_{T};\mathbb{R}^{n}\right)  $, $\Delta N_{T-1}\in
L^{2}\left(  \mathcal{F}_{T};\mathbb{R}^{n}\right)  $ such that $\mathbb{E}%
\left[  \Delta N_{T-1}|\mathcal{F}_{T-1}\right]  =0$, $\mathbb{E}\left[
e_{i}^{\ast}\Delta N_{T-1}\left(  \Delta W_{T-1}\right)  ^{\ast}%
|\mathcal{F}_{T-1}\right]  =0$ and%
\begin{equation}
\eta+f\left(  T,\eta\right)  -\mathbb{E}\left[  \eta+f\left(  T,\eta\right)
|\mathcal{F}_{T-1}\right]  =Z_{T-1}\Delta W_{T-1}+\Delta N_{T-1}%
.\label{bsde_proof_1}%
\end{equation}
Moreover, $\Delta N_{T-1}$ is uniquely determined in this decomposition. For
fixed $i\in\left\{  1,2,...,n\right\}  $, premultiply the equation by
$e_{i}^{\ast}$, postmultiply the equation by $\left(  \Delta W_{T-1}\right)
^{\ast}$ and then take the $\mathcal{F}_{T-1}$ conditional expectation. This
yields that
\[
\mathbb{E}\left[  e_{i}^{\ast}\left(  \eta+f\left(  T,\eta\right)  \right)
\left(  \Delta W_{T-1}\right)  ^{\ast}|\mathcal{F}_{T-1}\right]  =e_{i}^{\ast
}Z_{T-1}%
\]
since $\mathbb{E}\left[  \Delta W_{T-1}\left(  \Delta W_{T-1}\right)  ^{\ast
}|\mathcal{F}_{T-1}\right]  =I$. Therefore, we get the unique $Z_{T-1}$ by%
\[
Z_{T-1}=\mathbb{E}\left[  \left(  \eta+f\left(  T,\eta\right)  \right)
\left(  \Delta W_{T-1}\right)  ^{\ast}|\mathcal{F}_{T-1}\right]
\]
and%
\[%
\begin{array}
[c]{ccl}%
\mathbb{E}\left[  \left\Vert Z_{T-1}\right\Vert ^{2}\right]   & \leq &
\mathbb{E}\left[  \mathbb{E}\left[  \left\vert \eta+f\left(  T,\eta\right)
\right\vert ^{2}|\mathcal{F}_{T-1}\right]  \mathbb{E}\left[  \left\vert
\left(  \Delta W_{T-1}\right)  \right\vert ^{2}|\mathcal{F}_{T-1}\right]
\right]  <\infty.
\end{array}
\]
It leads that $Y_{T-1}=\mathbb{E}\left[  \eta+f\left(  T,\eta\right)
|\mathcal{F}_{T-1}\right]  $ and $Y_{T-1}\in L^{2}\left(  \mathcal{F}%
_{T-1};\mathbb{R}^{n}\right)  $.

Then, by similar arguments as above, we can obtain the unique solution
$\left(  Y_{t},Z_{t},\Delta N_{t}\right)  \in L^{2}\left(  \mathcal{F}%
_{t};\mathbb{R}^{n}\right)  \times L^{2}\left(  \mathcal{F}_{t};\mathbb{R}%
^{n\times d}\right)  \times L^{2}\left(  \mathcal{F}_{t};\mathbb{R}%
^{n}\right)  $ for $t\in\left\{  0,1,...,T-2\right\}  .$ Moreover,%
\begin{align*}
Z_{t} &  =\mathbb{E}\left[  \left(  Y_{t+1}+f\left(  t+1,Y_{t+1}%
,Z_{t+1}\right)  \right)  \left(  \Delta W_{t}\right)  ^{\ast}|\mathcal{F}%
_{t}\right]  ,\\
Y_{t} &  =\mathbb{E}\left[  Y_{t+1}+f\left(  t+1,Y_{t+1},Z_{t+1}\right)
|\mathcal{F}_{t}\right]  .
\end{align*}
By taking the convention $N_{0}=0$ and letting $N_{t}=N_{0}+\sum_{s=0}%
^{t-1}\Delta N_{s}$, we have that (\ref{bsde_form1}) holds true for all
$t\in\left\{  0,1,...,T-1\right\}  $. Finally, since%
\begin{align*}
&  \mathbb{E}\left[  e_{i}^{\ast}N_{t}\left(  W_{t}\right)  ^{\ast
}|\mathcal{F}_{t-1}\right]  \\
&  =e_{i}^{\ast}\sum_{s=0}^{t-2}\Delta N_{s}\mathbb{E}\left[  \left(
W_{t}\right)  ^{\ast}|\mathcal{F}_{t-1}\right]  +\mathbb{E}\left[  e_{i}%
^{\ast}\Delta N_{t-1}\left(  W_{t-1}+\Delta W_{t-1}\right)  ^{\ast
}|\mathcal{F}_{t-1}\right]  \\
&  =e_{i}^{\ast}N_{t-1}\left(  W_{t-1}\right)  ^{\ast},
\end{align*}
we conclude that $N$ is strongly orthogonal to $W$. This completes the proof.
\end{proof}

\begin{proof}
[Proof of Theorem \ref{bsde_result_2}]Note that%
\[%
\begin{array}
[c]{cl}
& \Delta\left\langle Y_{t},Y_{t}\right\rangle \\
= & \left\langle Y_{t+1},\Delta Y_{t}\right\rangle +\left\langle \Delta
Y_{t},Y_{t}\right\rangle \\
= & \left\langle Y_{t+1},-f\left(  t+1,Y_{t+1},Z_{t+1}\right)  \right\rangle
+\left\langle Y_{t}+\Delta Y_{t},Z_{t}\Delta W_{t}+\Delta N_{t}\right\rangle
\\
& +\left\langle -f\left(  t+1,Y_{t+1},Z_{t+1}\right)  +Z_{t}\Delta
W_{t}+\Delta N_{t},Y_{t}\right\rangle \\
= & \left\langle Y_{t+1},-f\left(  t+1,Y_{t+1},Z_{t+1}\right)  \right\rangle
+\left\langle Y_{t},Z_{t}\Delta W_{t}+\Delta N_{t}\right\rangle \\
& +\left\langle -f\left(  t+1,Y_{t+1},Z_{t+1}\right)  +Z_{t}\Delta
W_{t}+\Delta N_{t},Z_{t}\Delta W_{t}+\Delta N_{t}\right\rangle \\
& +\left\langle -f\left(  t+1,Y_{t+1},Z_{t+1}\right)  ,Y_{t}\right\rangle
+\left\langle Y_{t},Z_{t}\Delta W_{t}+\Delta N_{t}\right\rangle \\
= & 2\left\langle Y_{t+1},-f\left(  t+1,Y_{t+1},Z_{t+1}\right)  \right\rangle
-\left\langle f\left(  t+1,Y_{t+1},Z_{t+1}\right)  ,f\left(  t+1,Y_{t+1}%
,Z_{t+1}\right)  \right\rangle \\
& +\left\langle Z_{t}\Delta W_{t},Z_{t}\Delta W_{t}\right\rangle +\left\langle
\Delta N_{t},\Delta N_{t}\right\rangle +2\left\langle Y_{t},Z_{t}\Delta
W_{t}+\Delta N_{t}\right\rangle +2\left\langle Z_{t}\Delta W_{t},\Delta
N_{t}\right\rangle .
\end{array}
\]
Then,%
\[%
\begin{array}
[c]{cl}
& \mathbb{E}\left[  \left\vert Y_{t+1}\right\vert ^{2}-\left\vert
Y_{t}\right\vert ^{2}\right]  =\mathbb{E}\left[  \Delta\left\langle
Y_{t},Y_{t}\right\rangle \right]  \\
= & -2\mathbb{E}\left[  \left\langle Y_{t+1},f\left(  t+1,Y_{t+1}%
,Z_{t+1}\right)  \right\rangle \right]  -\mathbb{E}\left[  \left\vert f\left(
t+1,Y_{t+1},Z_{t+1}\right)  \right\vert ^{2}\right]  \\
& +\mathbb{E}\left[  \left\Vert Z_{t}\right\Vert ^{2}+\left\vert \Delta
N_{t}\right\vert ^{2}\right]  .
\end{array}
\]
Thus,%
\begin{align*}
&  \mathbb{E}\left[  \left\vert Y_{t}\right\vert ^{2}+\left\Vert
Z_{t}\right\Vert ^{2}+\left\vert \Delta N_{t}\right\vert ^{2}\right]  \\
&  =\mathbb{E}\left[  \left\vert Y_{t+1}\right\vert ^{2}\right]
+2\mathbb{E}\left[  \left\langle Y_{t+1},f\left(  t+1,Y_{t+1},Z_{t+1}\right)
\right\rangle \right]  +\mathbb{E}\left[  \left\vert f\left(  t+1,Y_{t+1}%
,Z_{t+1}\right)  \right\vert ^{2}\right]  \\
&  \leq2\mathbb{E}\left[  \left\vert Y_{t+1}\right\vert ^{2}\right]
+2\mathbb{E}\left[  \left\vert f\left(  t+1,Y_{t+1},Z_{t+1}\right)  -f\left(
t+1,0,0\right)  +f\left(  t+1,0,0\right)  \right\vert ^{2}\right]  \\
&  \leq C\mathbb{E}\left[  \left\vert Y_{t+1}\right\vert ^{2}+\left\Vert
Z_{t+1}\right\Vert ^{2}+\left\vert f\left(  t+1,0,0\right)  \right\vert
^{2}\right]  .
\end{align*}
Since $Y_{T}=\eta$ and $f\left(  T,y,z\right)  $ is independent of $z$, we
obtain that for any $t\in\{0,1,...,T-1\}$,%
\[
\mathbb{E}\left[  \left\vert Y_{t}\right\vert ^{2}+\left\Vert Z_{t}\right\Vert
^{2}+\left\vert \Delta N_{t}\right\vert ^{2}\right]  \leq C\mathbb{E}\left[
\left\vert \eta\right\vert ^{2}+\sum_{s=1}^{T}\left\vert f\left(
s,0,0\right)  \right\vert ^{2}\right]
\]
by induction. Then the result can be proved obviously. This completes the proof.
\end{proof}

\end{document}